\documentclass[12pt,a4paper]{amsart}
\usepackage{latexsym,amsfonts,amsthm,amsmath,mathrsfs,amssymb}
\usepackage[british]{babel}
\usepackage[dvips]{color}
\usepackage[utf8]{inputenc}
\usepackage[T1]{fontenc}
\usepackage[dvips,final]{graphics}
\usepackage{xcolor}
\usepackage{mathrsfs}
\usepackage{mathtools} 
\usepackage{breqn}
\usepackage{epstopdf}
\usepackage{verbatim}
\usepackage{amscd}
\usepackage{hyperref}
\usepackage[english,capitalize]{cleveref}
\usepackage{aliascnt}

\usepackage{todonotes}

\makeatletter
\def\newaliasedtheorem#1[#2]#3{
	\newaliascnt{#1@alt}{#2}
	\newtheorem{#1}[#1@alt]{#3}
	\expandafter\newcommand\csname #1@altname\endcsname{#3}
}
\makeatother

\theoremstyle{plain}

\newtheorem{theorem}{Theorem}[section]
\newaliasedtheorem{prop}[theorem]{Proposition}
\newaliasedtheorem{lem}[theorem]{Lemma}
\newaliasedtheorem{coroll}[theorem]{Corollary}

\theoremstyle{definition}
\newaliasedtheorem{defi}[theorem]{Definition}
\newaliasedtheorem{quest}[theorem]{Question}
\newaliasedtheorem{fact}[theorem]{Fact}

\theoremstyle{remark}
\newaliasedtheorem{rem}[theorem]{Remark}
\newaliasedtheorem{exa}[theorem]{Example}

\newcommand{\R}{\mathbb{R}}

\newcommand{\C}{\mathbb{C}}

\let\altphi\phi
\let\phi\varphi
\let\varphi\altphi
\let\altphi\undefined

\newcommand{\average}{{\mathchoice {\kern1ex\vcenter{\hrule height.4pt
width 6pt
depth0pt} \kern-9.7pt} {\kern1ex\vcenter{\hrule height.4pt width 4.3pt
depth0pt}
\kern-7pt} {} {} }}

\address{\textsc{Daniela Di Donato}: 
Dipartimento di Ingegneria Industriale e Scienze Matematiche, Via Brecce Bianche, 12 60131 Ancona, Universit\'a Politecnica delle Marche.}
\email{daniela.didonato@unitn.it}

\title{A note about Intrinsically Lipschitz constants}

\date{\today}

\author{ Daniela Di Donato}

\begin{document}

\begin{abstract}
		Recently, Le Donne and the author introduce the notion of  intrinsically Lipschitz sections in metric spaces. The idea of this paper is to investigate about the properties of the intrinsically Lipschitz constants.	More precisely, we give the Leibniz formula  and the product formula for the intrinsic slope.
\end{abstract}

\maketitle 
\tableofcontents

\section{Introduction} 
In \cite{DDLD21}, Le Donne and the author introduce a new concept of Lipschitz graphs in general metric spaces. Our start point is to consider intrinsically Lipschitz maps introduced by Franchi, Serapioni and Serra Cassano \cite{FSSC, FSSC03, MR2032504} (see also  \cite{SC16, FS16}) in subRiemannian Carnot groups \cite{ABB, BLU, CDPT}. This concept was born in order to give a good definition of rectifiability in subRiemannian geometry after the negative result shown in \cite{AmbrosioKirchheimRect}. The notion of rectifiable set is a key one in Calculus of Variations and in Geometric Measure Theory.

In \cite{DDLD21} there are relevant results like Ahlfors-David regularity, the Ascoli-Arzel\'a Theorem, the Extension Theorem for the so-called intrinsically Lipschitz sections. Our approach is to consider graph theory instead of map one. Yet, in  \cite{D22.1, D22.2, D22.4} the author introduced and studied other natural definitions  in metric spaces: the  intrinsically H\"older, quasi-isometric  and quasi-symmetric sections, respectively.

In this note we focus our attention on the intrinsically Lipschitz constants. The long-term objective is to adapt Cheeger theory \cite{C99} in our intrinsic context. Following  \cite{DM}, we prove the Leibniz formula (see Proposition \ref{propLeibnitz formula for slope}) and the product formula (see Proposition \ref{propslope.9apr}) for intrinsic slope of Lipschitz sections.

We begin recall the definition of intrinsically Lipschitz sections given in \cite{DDLD21}. Our setting is the following. We have a metric space $X$, a topological space $Y$, and a 
quotient map $\pi:X\to Y$, meaning
continuous, open, and surjective.
The standard example for us is when $X$ is a metric Lie group $G$ (meaning that the Lie group $G$ is equipped with a left-invariant distance that induces the manifold topology), for example a subRiemannian Carnot group, 
and $Y$ if the space of left cosets $G/H$, where 
$H<G$ is a  closed subgroup and $\pi:G\to G/H$ is the projection modulo $H$, $g\mapsto gH$.

\begin{defi}[Intrinsic Lipschitz section]\label{Intrinsic Lipschitz section}
Let $(X,d)$ be a metric space and let $Y$ be a topological space. We say that a map $\phi :Y \to X$ is a {\em section} of a quotient map $\pi :X \to Y$ if
\begin{equation*}
\pi \circ \phi =\mbox{id}_Y.
\end{equation*}
Moreover, we say that $\phi$ is an {\em intrinsically Lipschitz section} with constant $L\geq 1$ if in addition
\begin{equation*}
d(\phi (y_1), \phi (y_2)) \leq L d(\phi (y_1), \pi ^{-1} (y_2)), \quad \mbox{for all } y_1, y_2 \in Y.
\end{equation*}

Equivalently, we are requesting that  that
\begin{equation*}
d(x_1, x_2) \leq L d(x_1, \pi ^{-1} (\pi (x_2))), \quad \mbox{for all } x_1,x_2 \in \phi (Y) .
\end{equation*}
\end{defi}

We underline that, in the case  $
 \pi$ is a Lipschitz quotient or submetry \cite{MR1736929, Berestovski}, the results trivialize, since in this case being intrinsically Lipschitz  is equivalent to biLipschitz embedding, see Proposition 2.4 in \cite{DDLD21}. 

The rest of the paper is organized as follows. In $\mathbf{Section \, 2}$ we recall some basic definition of intrinsically Lipschitz section in order to show that a suitable subset (see Definition \ref{defwrtpsinew.9apr}) is a convex set (see Corollary \ref{corollIMPO}).   In $\mathbf{Section \, 3}$ we provide some basic properties of the intrinsic Lipschitz constants like the Leibniz formula (see Proposition \ref{propLeibnitz formula for slope}) and the product formula (see Proposition \ref{propslope.9apr}).   
$\mathbf{Section \, 4}$ states maximum, minimum and inverse of intrinsic Lipschitz sections are so too.
In $\mathbf{Section \, 5},$ we investigate when the class of intrinsically Lipschitz sections is a convex set and we give a "stronger" version of  the Leibniz and the product formula 	with additional hypothesis.

\section{Intrinsic Lipschitz set with respect to another one}
In \cite{DDLD21}, we introduce the notion of intrinsically Lipschitz with respect to another intrinsically Lipschitz section as follows.

%

\begin{defi}[Intrinsic Lipschitz  with respect to  a section]\label{defwrtpsinew}
 Given  sections 
  $\phi, \psi :Y\to X$   
  of $\pi$. We say that   $\phi $ is {\em intrinsically $L$-Lipschitz with respect to  $\psi$ at point $\hat x$}, with $L\geq1$ and $\hat x\in X$, if
\begin{enumerate}
\item $\hat x\in \psi(Y)\cap \phi (Y);$
\item $\phi  (Y) \cap  C_{\hat x,L}^{\psi} = \emptyset ,$
\end{enumerate}
where
$$ C_{\hat x,L}^{\psi} := \{x\in X \,:\,  d(x, \psi (\pi (x))) > L d(\hat x, \psi (\pi (x)))  \}.   $$
\end{defi}

  \begin{rem} Definition~\ref{defwrtpsinew} can be rephrased as follows.
 A section $\phi  $ is intrinsically $L$-Lipschitz with respect to  $\psi$ at point $\hat x$ if
 and only if 
 there is $\hat y\in Y$ such that  $\hat x= \phi (\hat y)=\psi(\hat y)$ and
\begin{equation}\label{defintrlipnuova}
 d(x, \psi (\pi (x))) \leq L d(\hat x, \psi (\pi (\hat x))), \quad \forall x \in \phi (Y), 
\end{equation}
which equivalently means 
\begin{equation}\label{equation28.0}
 d(\phi (y), \psi (y)) \leq L d(\psi(\hat y), \psi (y)) ,\qquad \forall y\in Y. 
\end{equation}
  \end{rem}  

  \begin{rem} We stress that Definition~\ref{defwrtpsinew} does not induce  an equivalence relation, because of lack of symmetry in the right-hand side of \eqref{equation28.0}. On the other hand, in  \cite[Theorem 4.2]{D22.1} the author introduce a stronger condition in order to obtain an equivalence relation. 
  \end{rem}

  Now we are able to define the key notion of this paper.
    \begin{defi}[Intrinsic Lipschitz set with respect to $\psi$]\label{defwrtpsinew.9apr} Let  $\psi: Y \to X$ a section of $\pi$.  We define the set  of all  intrinsically $L$-Lipschitz section of $\pi$ with respect to  $\psi$ at point $\hat x$ as
\begin{equation*}
\begin{aligned}
ILS _{\psi , \hat x} & :=\{ \phi :Y\to X \mbox{ section of $\pi$ } :\, \phi \mbox{ is intrinsically $\tilde L$-Lipschitz w.r.t. $\psi$ at point $\hat x$ for}\\
& \quad \quad \mbox{ some $\tilde L\geq 1$} \}.
\end{aligned}
\end{equation*}
  \end{defi}
  
An interesting observation is that, considering $ILS _{\psi , \hat x},$ the intrinsic Lipschitz constant $L$ can be change but it is fundamental that the point $\hat x$ is a common one for the every sections. Regarding  $ILS _{\psi , \hat x},$  in \cite[Theorem 3.5]{D22.1} we have the following result.

   \begin{theorem}\label{theorem} Let $\pi :X \to Y$ is a linear and quotient map from a normed space $X$ to a metric space $Y.$ Assume also that $\psi :Y \to X$ is a section of $\pi$ with $L \geq 1$  and $\{\lambda \hat x\, :\, \lambda \in \R ^+\} \subset X$ with $\hat x \in \psi (Y).$
   
Then, the set $\bigcup _{\lambda \in \R ^+} ILS _{\lambda \psi , \lambda \hat x} \cup \{ 0 \}$ is a vector space over $\R $ or $\C .$
       \end{theorem}
      
Notice that  it is no possible to obtain the statement for $ ILS _{\psi , \hat x}$ since the simply observation that if $\psi (\hat y) = \hat x$ then $\psi (\hat y ) + \psi (\hat y )  \ne \hat x.$ On the other hand, in our idea $\pi$ is the 'usual' projection map and so it is not too restrictive to ask its linearity.  We conclude the introduction given a link between the intrinsically Lipschitz sections and the intrinsic sections with respect to another one.
 \begin{prop}\cite[Proposition 1.5]{DDLD21}\label{linkintrinsicocneelip}
   Let $X  $ be a metric space, $Y$ a topological space, $\pi :X \to Y$   a  quotient map, and $L\geq 1$.
Assume that   every point $x\in X$ is contained in the image of an intrinsic $L$-Lipschitz section $\psi_x$ for $\pi$.
 Then for every section $\phi :Y\to X$ of $\pi$ the following are equivalent:
   \begin{enumerate}
\item for all $x\in\phi(Y)$ the section $\phi $ is intrinsically $L_1$-Lipschitz with respect to  $\psi_x$ at   $x;$
\item  the section $\phi $  is intrinsically $L_2$-Lipschitz.
\end{enumerate}
   \end{prop}

\section{Intrinsic   Lipschitz  constants and their properties}\label{Intrinsic}
In this section we introduce  intrinsically Lipschitz constants and then we prove the following statement.  Regarding the normed space theory, the reader can see  \cite{B10, M98}. 
 \begin{prop}\label{theorem2}
   Let $X$ be a normed space, $Y$ be a metric space and $\pi :X \to Y$  be a  quotient map.
Assume that   every point $x\in X$ is contained in the image of an intrinsic $L$-Lipschitz section $\psi_x$ for $\pi$ and that the section $\phi :Y \to X$ is intrinsically $L_1$-Lipschitz with respect to  $\psi_x$ at   $x$. Then for any $\bar y \in Y$ such that $\phi (\bar y)=x$ the following are true.
\begin{enumerate}
\item if $f\in C(Y, [0,1])$ then denoting $\eta = f\phi + (1-f) \psi $ the map $Y\to X$ we have that
\begin{equation*}\label{equationLeibniz.0}
Ils (\eta )(\bar y) \leq f(\bar y) Ils (\phi) (\bar y)+ (1-f(\bar y)) Ils (\psi) (\bar y).
\end{equation*}
\item If it holds
 \begin{equation*}
d(\psi^2 (\bar y), \pi ^{-1} (y)) \geq d(\psi (\bar y), \pi ^{-1} (y)) , \quad \forall y\in Y
\end{equation*}
then
\begin{equation*}\label{equationLeibniz.prod.0}
Ils (\phi \psi )(\bar y) \leq \sup |\psi | Ils (\phi) (\bar y)+ \sup |\phi | Ils (\psi) (\bar y).
\end{equation*}
\end{enumerate}
   \end{prop}

\subsection{Intrinsic   Lipschitz  constants}
 We adapt the theory of  \cite{C99, DM} in our intrinsic case.

\begin{defi}\label{def_ILS.1} Let $\phi:Y\to X$ be a section of $\pi$. Then we define 
\begin{equation*}
ILS (\phi):= \sup _{\substack{y_1, y_2 \in Y \\ y_1\ne y_2}} \frac{d(\phi (y_1), \phi (y_2))}{  d(\phi (y_1), \pi ^{-1} (y_2)) } \in [1,\infty ]
\end{equation*}
and
\begin{equation*}
\begin{aligned}
ILS (Y,X,\pi ) &:= \{ \phi :Y \to X \,:\, \phi \mbox{ is an intrinsically Lipschitz section of $\pi$ and }  ILS(\phi) < \infty \},\\
ILS_{b} (Y,X,\pi) & := \{ \phi \in  ILS (Y,X,\pi) \,:\, \mbox{spt}(\phi) \mbox{ is bounded} \}.\\
\end{aligned}
\end{equation*}
For simplicity, we will write $ILS (Y,X)$ instead of  $ILS (Y,X,\pi ).$
\end{defi}

\begin{defi}\label{def_ILS.2} Let $\phi:Y\to X$ be a  section of $\pi$. Then we define the local intrinsically Lipschitz constant (also called slope) of $\phi$ the map $Ils:Y \to [1,+\infty )$ defined as 
\begin{equation*}
Ils (\phi) (z):= \limsup _{y\to z} \frac{d(\phi (y), \phi (z))}{  d(\phi (y), \pi ^{-1} (z)) },
\end{equation*}
if $z \in Y$ is an accumulation point; and $Ils (\phi) (z):=0$ otherwise.
\end{defi}

\begin{defi}\label{def_ILS.3} Let $\phi:Y\to X$ be a section of $\pi$. Then we define the asymptotic  intrinsically Lipschitz constant of $\phi$ the map $Ils_a:Y \to [1,+\infty )$ given by
\begin{equation*}
Ils_a (f) (z):= \limsup _{y_1,y_2\to z}\frac{d(\phi (y_1),\phi (y_2))}{  d(\phi (y_1), \pi ^{-1} (y_2)) }
\end{equation*}
if $z \in Y$ is an accumulation point and $Ils (\phi) (z):=0$ otherwise.
\end{defi}
 
\begin{rem}\label{defrem} Notice that by $\phi (y_2) \in \pi ^{-1} (y_2),$ it is trivial that $ d(\phi (y_1), \pi ^{-1} (y_2)) \leq d(\phi (y_1), \phi (y_2))$ and so $Ils(\phi) \geq 1.$ Moreover, it holds 
\begin{equation*}
Ils(\phi) \leq Ils_a (\phi ) \leq ILS(\phi).
\end{equation*}
\end{rem}

\subsection{Leibniz formula for the slope} Using Leibniz formula for the intrinsic slope, it is possible to give a convex set of intrinsically Lipschitz sections for the set $ILS _{\psi , \hat x}$ given by Definition \ref{defwrtpsinew.9apr} (see Corollary \eqref{corollIMPO}).
 \begin{prop}[Leibniz formula for the slope]\label{propLeibnitz formula for slope} 
Let $X$ be a normed space, $Z$ be an open set of a metric space $Y$ and let $\phi , \psi  \in ILS_{loc}(Z,X)$ such that $\phi (\bar y)= \psi (\bar y)$ for some $\bar y\in Z$ and $f\in C(Z, [0,1]).$  Then denoting $\eta = f\phi + (1-f) \psi $ the map $Z\to X$ we have that
\begin{equation}\label{equationLeibnitzNEW}
Ils (\eta )(\bar y) \leq f(\bar y) Ils (\phi) (\bar y)+ (1-f(\bar y)) Ils (\psi) (\bar y).
\end{equation}

\end{prop}

 \begin{proof} 
Notice that $\eta (\bar y)=\phi (\bar y) = \psi (\bar y)$ and that for every $y \in Y$ we have
\begin{equation*}
\begin{aligned}
\|\eta (y)-\eta (\bar y)\|  & =\| f(y)(\phi (y)- \phi (\bar y)) + (1-f(y)) (\psi (y)-\psi(\bar y))\|
\end{aligned}
\end{equation*}
and so
\begin{equation*}
\begin{aligned}
\|\eta (y)-\eta (\bar y)\| & \leq f(y) \|\phi (y)-\phi (\bar y)\|+ (1-f(y)) \|\psi (y)-\psi (\bar y)\|,
\end{aligned}
\end{equation*}
Hence, dividing for  $d(\eta (\bar y), \pi ^{-1} (y))$    we obtain
 \begin{equation*}
\begin{aligned}
\frac{d(\eta (y),\eta (\bar y))}{d(\eta (\bar y), \pi ^{-1} (y))}& \leq  f(y) \frac{   d(\phi (y), \phi (\bar y))}{d(\eta  (\bar y), \pi ^{-1} (y))}  + 
(1-f(y))\frac{ d(\psi (y),\psi (\bar y))}{d(\eta  (\bar y), \pi ^{-1} (y))}   \\
\end{aligned}
\end{equation*}
Now taking the supremum in $y\in B(\bar y , r)$ on the left hand side and letting $r\to 0$ we get the thesis \eqref{equationLeibnitzNEW}.
\end{proof}

 \begin{coroll} 
Let $X$ be a normed space, $Z$ be an open set of a metric space $Y$ and let $\phi , \psi  \in ILS_{loc}(Z,X)$ such that $\phi (\bar y)= \psi (\bar y)$ for some $\bar y\in Z.$ Then 
\begin{equation*} 
Ils (\alpha \phi + \beta \psi ) (\bar y) \leq |\alpha | Ils ( \phi ) (\bar y)+ |\beta | Ils ( \psi ) (\bar y ), 
\end{equation*}
for any $\alpha , \beta \in \R$ such that $\alpha +\beta =1.$
\end{coroll}

 \begin{coroll}\label{convexity} 
Under the same assumption of Proposition \ref{propLeibnitz formula for slope}, we have that a convex combination of $\phi$ and $\psi$ is also an intrinsically Lipschitz section of $\pi$ at any point $\bar y\in Y$ such that $\phi (\bar y)=\psi (\bar y).$
\end{coroll}

We conclude this section given an immediately  consequence of Corollary \ref{convexity} and Proposition \ref{linkintrinsicocneelip}.
 \begin{coroll}\label{corollIMPO}
  Let $\pi :X \to Y$ be a quotient map from a normed space $X$ to a metric space $Y.$ Assume also that $\psi :Y \to X$ is an intrinsically $L$-Lipschitz section of $\pi$ with $L \geq 1$ and $\hat x \in X.$ Then, the set $ILS _{\psi , \hat x}$ is a convex set.
\end{coroll}

\subsection{Product for the slope} An important point in order to get a relation between the intrinsic slope of a section and its square is the following result.
 \begin{prop}[Product of the slope]\label{propslope.9apr}
Let $X$ be a normed space, $Z$ be an open set of a metric space $Y$ and let $\phi , \psi  \in ILS_{b}(Z,X)$ such that $\phi (\bar y)= \psi (\bar y)$ for some $\bar y\in Z.$ We also assume that
\begin{equation*}
d(\phi^2 (\bar y), \pi ^{-1} (y)) \geq d(\phi (\bar y), \pi ^{-1} (y)) , \quad \forall y\in Y.
\end{equation*}

Then
\begin{equation}\label{equationLeibnitz.prod}
Ils (\phi \psi )(\bar y) \leq \sup |\psi | Ils (\phi) (\bar y)+ \sup |\phi | Ils (\psi) (\bar y).
\end{equation}

\end{prop}

 \begin{proof} 
Notice that  for every $y \in Y$ we have
\begin{equation*}
\begin{aligned}
d(\phi (y) \psi (y), \phi(\bar y) \psi (\bar y))& \leq d(\phi (y) \psi (y), \phi( y) \psi (\bar y)) + d(\phi (y) \psi (\bar y), \phi(\bar y) \psi (\bar y))\\
& \leq \sup |\phi | d( \psi (y), \psi (\bar y)) + |\psi (\bar y)| d(\phi (y), \phi(\bar y))\\
\end{aligned}
\end{equation*}
Hence, recall that $\phi (\bar y)=\psi (\bar y),$ dividing for  $d(\phi ^2(\bar y), \pi ^{-1} (y))$    we deduce that
 \begin{equation*}
\begin{aligned}
\frac{d(\phi (y) \psi (y), \phi(\bar y) \psi (\bar y)) }{d( \phi ^2(\bar y), \pi ^{-1} (y))}& \leq  \sup |\phi |  \frac{   d( \psi (y), \psi (\bar y)) }{d(\phi ^2 (\bar y), \pi ^{-1} (y))}  +   \psi (\bar y ) \frac{  d(\phi (y), \phi(\bar y))  }{d(\phi ^2 (\bar y), \pi ^{-1} (y))} \\
& \leq  \sup |\phi |  \frac{   d( \psi (y), \psi (\bar y)) }{d(\phi  (\bar y), \pi ^{-1} (y))}  +   \psi (\bar y ) \frac{  d(\phi (y), \phi(\bar y))  }{d(\psi  (\bar y), \pi ^{-1} (y))} \\
\end{aligned}
\end{equation*}
Now taking the supremum in $y\in B(\bar y , r)$ on the left hand side and letting $r\to 0$ we get the thesis \eqref{equationLeibnitz.prod}.
\end{proof}

 \begin{coroll} 
Let $X$ be a normed space, $Z$ be an open set of a metric space $Y$ and let  $\phi  \in ILS_{b}(Z,\R)$ such that 
$
d(\phi^2 (\bar y), \pi ^{-1} (y)) \geq d(\phi (\bar y), \pi ^{-1} (y)) , \, \forall y\in Y.
$
Then, $$ Ils (\phi ^2)(\bar y) \leq 2\sup |\phi | Ils (\phi) (\bar y).$$
\end{coroll}

 \begin{proof} 
The statement follows in a similar way of Proposition \ref{propslope.9apr}   noting that for any $y\in Y$
\begin{equation*}
\begin{aligned}
d(\phi ^2(y) , \phi ^2(\bar y) )& \leq  2 \sup |\phi |  d(\phi (y), \phi(\bar y)).\\
\end{aligned}
\end{equation*}
\end{proof}

\subsection{Proof of Proposition \ref{theorem2}} Noting that by Proposition \ref{linkintrinsicocneelip}, we have that $\phi$ is also an intrinsically Lipschitz section. Hence, the proof of Proposition \ref{theorem2} is a trivial consequence of Proposition \ref{propLeibnitz formula for slope} and \ref{propslope.9apr}.

\section{Maximum, minimum and inverse of intrinsically Lipschitz  sections}
The next proposition summarizes the basic properties of real valued intrinsically Lipschitz sections. Here, for any sections $\phi , \psi :Y \to \R$ of $\pi :\R \to Y,$ we define 
\begin{equation*}
\begin{aligned}
(\phi \vee \psi )(y)& = \max\{\phi(y), \psi (y)\},   \quad \forall y\in Y,\\
(\phi \wedge \psi )(y)&=\min\{\phi(y), \psi (y)\}, \quad \forall y\in Y.
\end{aligned}
\end{equation*}

 \begin{prop}
   Let $Y$ be a topological space and $\phi , \psi :Y \to \R$ be intrinsically Lipschitz sections of $\pi$ with constants $L_\phi$ and $L_\psi$ with $L_\phi, L_\psi \geq 1$. Then,
\begin{enumerate}
\item If $\phi (y) \geq \varepsilon >0$ for any $y\in Y,$ then $1/\phi$ is an intrinsically $M$-Lipschitz section with $M = 1/\varepsilon ^2 L_\phi.$
\item $\phi \vee \psi$ and $\phi \wedge \psi$ are intrinsically Lipschitz sections with $L (\phi \vee \psi) , L(\phi \wedge \psi) = \max\{L_\phi , L_\psi \}$.
\end{enumerate}

   \end{prop}

 \begin{proof}
(1). This follows noting that for any $y_1, y_2 \in Y$
\begin{equation*}
\begin{aligned}
\left|\frac 1{\phi (y_1)} - \frac 1{\phi (y_2)}  \right| = \frac {|\phi (y_1)-\phi (y_2)|}{|\phi (y_1) \phi (y_2)|} \leq  \frac {L_\phi }{\varepsilon ^2} d(\phi (y_1), \pi^{-1}(y_2)).
\end{aligned}
\end{equation*}

(2). Let $\eta =\phi \vee \psi $ and fix $y_1, y_2 \in Y.$ If $\eta (y_i)=\phi(y_i)$ or $\eta (y_i)=\psi(y_i)$ for $i=1,2$ the statement is trivial. Hence, we suppose that $\eta (y_1)=\phi (y_1), \eta (y_2)=\psi (y_2)$ and $\phi (y_1)<\psi (y_2).$ Then,
\begin{equation*}
\begin{aligned}
\eta (y_1)-\eta (y_2) = \phi (y_1)-\psi (y_2) \leq \phi (y_1)-\phi (y_2) \leq L_\phi d(\phi (y_1),\pi^{-1}(y_2)) .
\end{aligned}
\end{equation*}
If $\phi (y_1)>\psi (y_2),$ we get the inequality with $L_\psi$ instead of $L_\phi$. By a similar argument, using the formula $\phi \vee \psi = -((-\phi) \wedge (-\psi))$, we deduce the statement from $\phi \wedge  \psi .$   The proof of statement is complete.
   \end{proof}

\section{A suitable convex set for intrinsically Lipschitz sections} In the last sections we get the Leibniz formula and the product formula at a special point $\hat x\in X.$ Here, we obtain these results for any point with additional hypothesis. Moreover, we build a convex set of intrinsically Lipschitz sections (see Corollary \ref{corollCONVEXD}).

 \begin{prop}[Leibniz formula for the slope (stronger version)]\label{propLeibnitz formula for slope.27} 
Let $X$ be a normed and  convex space, $Y$ be  a metric space and let $\phi $ and $ \psi $ be  intrinsically $L$-Lipschitz sections of $\pi$ such that
\begin{enumerate}
\item  $Im(\phi)=Im(\psi).$
\item it holds
\begin{equation*}
 \frac{  d(f  (z_1), \pi ^{-1} (y)) }{d(f(z_2) , \pi ^{-1} (y))} \leq \ell <\infty, \quad  \mbox{ for } f=\phi , \psi , \,\, \forall z_1,z_2,y \in Y.
\end{equation*}
\end{enumerate}
 Denoting $\eta = t \phi + (1-t) \psi $ the map $Y\to X$ with $t\in [0,1]$  we have that
\begin{equation}\label{equationLeibnitz}
Ils (\eta )( y) \leq t \ell Ils (\phi) ( y)+ (1-t)\ell Ils (\psi) ( y),\quad \forall y\in Y.
\end{equation}

\end{prop}

 \begin{proof} 
In the similar way to Proposition \ref{propLeibnitz formula for slope}, we have that for any $y, z\in Y$
\begin{equation*}
\begin{aligned}
\|\eta (y)-\eta (z)\| & \leq t \|\phi (y)-\phi (z)\|+ (1-t)\|\psi (y)-\psi (z)\|.
\end{aligned}
\end{equation*}
Hence, by convexity of $X,$ $Im(\eta) =Im(\phi)=Im(\psi)$ and so dividing for  $d(\eta (z), \pi ^{-1} (y))$    we obtain
 \begin{equation*}
\begin{aligned}
\frac{d(\eta (y),\eta (z))}{d(\eta (z), \pi ^{-1} (y))}& \leq  t \frac{   d(\phi (y), \phi (z))}{d(\phi  (z), \pi ^{-1} (y))}  \frac{  d(\phi  (z), \pi ^{-1} (y)) }{d(\eta (z) , \pi ^{-1} (y))}   + 
(1-t) \frac{ d(\psi (y),\psi (z))}{d(\psi (z), \pi ^{-1} (y))}   \frac{  d(\psi (z), \pi ^{-1} (y)) }{d(\eta  (z), \pi ^{-1} (y))}  \\
& = t \frac{   d(\phi (y), \phi (z))}{d(\phi  (z), \pi ^{-1} (y))}  \frac{  d(\phi  (z), \pi ^{-1} (y)) }{d(\phi (z_1) , \pi ^{-1} (y))}   + 
(1-t)\frac{ d(\psi (y),\psi (z))}{d(\psi (z), \pi ^{-1} (y))}   \frac{  d(\psi (z), \pi ^{-1} (y)) }{d(\psi (z_2), \pi ^{-1} (y))}  \\
& \leq  t \ell \frac{   d(\phi (y), \phi (z))}{d(\phi  (z), \pi ^{-1} (y))}   + 
(1-t)\ell \frac{ d(\psi (y),\psi (z))}{d(\psi (z), \pi ^{-1} (y))}    \\
\end{aligned}
\end{equation*}
Now taking the supremum in $z\in B( y , r)$ on the left hand side and letting $r\to 0$ we get the thesis \eqref{equationLeibnitz}.
\end{proof}

 \begin{coroll}\label{corollCONVEXD}  
Let $X$ be a normed and  convex space, $Y$ be  a metric space. The following set is convex one: the set of all   intrinsically $L$-Lipschitz sections of $\pi$ such that
\begin{enumerate}
\item  $Im(\phi)=Im(\psi).$
\item it holds
\begin{equation*}
 \frac{  d(\phi  (z_1), \pi ^{-1} (y)) }{d(\phi (z_2) , \pi ^{-1} (y))} \leq \ell <\infty, \quad \forall z_1,z_2,y \in Y.
\end{equation*}
\end{enumerate}
\end{coroll}

 \begin{prop}[Product of the slope (stronger version)]\label{propslope.9apr.27}
Let $X$ be a normed space, $Y$ a metric space and let $\phi $ be an intrinsically Lipschitz section of $\pi$ bounded by $M$ such that
\begin{equation*}
d(\phi^2 (z), \pi ^{-1} (y)) \gtrsim d(\phi (z), \pi ^{-1} (y)) , \quad \forall y,z \in Y.
\end{equation*}
Then
\begin{equation}\label{equationLeibnitz.prod.27}
Ils (\phi ^2 )( y) \leq  2M Ils (\phi) ( y), \quad \forall y\in Y.
\end{equation}

\end{prop}

 \begin{proof} 
Fix $y\in Y.$ As Proposition \ref{propslope.9apr}, we get for any $z\in Y$ 
 \begin{equation*}
\begin{aligned}
\frac{d(\phi ^2 (y), \phi ^2(z)) }{d( \phi ^2(z), \pi ^{-1} (y))}& \leq  2M\frac{  d(\phi (y), \phi(z))  }{d(\phi  (z), \pi ^{-1} (y))},\\
\end{aligned}
\end{equation*}
Now taking the supremum in $z\in B(y, r)$ on the left hand side and letting $r\to 0$ we get the thesis \eqref{equationLeibnitz.prod.27}.
\end{proof}

 \bibliographystyle{alpha}
\bibliography{DDLD}

\end{document}